\documentclass[12pt]{article}

\usepackage[left=2.5cm, right=2.5cm, top=3cm, bottom=3cm]{geometry}

\usepackage{hyperref}
\usepackage{url}
\usepackage{amsthm,amsmath,amssymb}
\usepackage{amsfonts,dsfont} 
\usepackage{mathtools} 
\usepackage{graphics}
\usepackage{epsfig}
\usepackage{tablefootnote}
\usepackage{cite}
\usepackage{authblk}
\newtheorem{theo}{Theorem}[section]  

\newtheorem{lemm}[theo]{Lemma}
\newtheorem{rema}[theo]{Remark}

\newtheorem{assumption}[theo]{Assumption}

\newcommand{\beq}{\begin{equation}}
\newcommand{\eeq}{\end{equation}}
\newcommand{\beqa}{\begin{eqnarray}}
\newcommand{\eeqa}{\end{eqnarray}}
\newcommand{\beqas}{\begin{eqnarray*}}
\newcommand{\eeqas}{\end{eqnarray*}}
\newcommand{\beqs}{\begin{equation*}}
\newcommand{\eeqs}{\end{equation*}}

\newcommand{\R}{\mathbb R}

\newcommand{\hesslb}{{\mu}}

\newcommand{\hessub}{L}
\newcommand{\hessubi}{{L_i}}

\newcommand{\dist}{\mbox{dist}}
\newcommand{\tik}{\tau_i^k}

\title{Convergence rate of incremental \\aggregated gradient algorithms}
\author[*]{M. G\"urb\"uzbalaban $^*$, A. Ozdaglar $^*$, P. Parrilo \thanks{Laboratory for Information and Decision Systems, Massachusetts Institute of Technology, Cambridge, MA 02140, USA. emails: mertg@mit.edu, asuman@mit.edu, parrilo@mit.edu.}}

\date{\today}
\begin{document}

\maketitle

\begin{abstract} 
Motivated by applications to distributed optimization over networks and large-scale data processing in machine learning, we analyze the deterministic incremental aggregated gradient method for minimizing a finite sum of smooth functions where the sum is strongly convex. This method processes the functions one at a time in a deterministic order and incorporates a memory of previous gradient values to accelerate convergence.  
Empirically it performs well in practice; however, no theoretical analysis with explicit rate results was previously given in the literature to our knowledge, in particular most of the recent efforts concentrated on the randomized versions. In this paper, we show that this deterministic algorithm has global linear convergence and characterize the convergence rate. We also consider an aggregated method with momentum and demonstrate its linear convergence. Our proofs rely on a careful choice of a Lyapunov function that offers insight into the algorithm's behavior and simplifies the proofs considerably. 
%
\end{abstract}

\section{Introduction}
\label{sec-intro}
We consider the following unconstrained optimization problem where the objective function is the sum of component functions:
\beqa\label{pbm-multi-agent}
	\min_{x \in \R^n} f(x) = \sum_{i=1}^m f_i(x) 
\eeqa
where each $f_i: \R^n \to \R$ is a convex, continuously differentiable function referred to as a \textit{component function}. This problem arises in many applications including least square problems \cite{AlgEkfs2003,Bertsekas1996incremental} or more general parameter estimation problems where $f_i$ is the corresponding loss function of the $i$-th data block \cite{GurOzPar15}, distributed optimization in wireless sensor networks \cite{Blatt2007incremental}, machine learning problems \cite{Leroux2012sgd,BottouLecun2005} and minimization of expected value of a function (where the expectation is taken over  a finite probability distribution or approximated by an $m$-sample average). 

One widely studied approach is the (deterministic) incremental gradient (IG) method, which cycles through the component functions using a cyclic order and updates the iterates using the gradient of a single component function one at a time \cite{Bertsekas99nonlinear}. This method can be faster than non-incremental methods since each step is relatively cheaper (one gradient computation instead of $m$ gradient computations in the non-incremental case) and each step makes reasonable progress on average \cite{Bertsekas99nonlinear}. However, IG requires the stepsize to go to zero to obtain convergence to an optimal solution of problem \eqref{pbm-multi-agent} even if it is applied to smooth and strongly convex component functions \cite{bertsekas2011incremental}, unless a restrictive ``gradient growth condition" holds \cite{Solodov98IncrGrad}. As a consequence, with a decaying stepsize, IG has typical sublinear convergence rate properties. The same observation applies both to stochastic gradient methods (which uses a random order for cycling through the component functions) \cite{schmidt2013fast} and to incremental Newton methods that are of second-order \cite{GurOzPar15}. 

Another interesting class of methods includes the  \textit{incremental aggregated gradient} (IAG) method of Blatt {\it et al.} (see \cite{Blatt2007incremental,TsengYun2014incremental}) and closely-related stochastic methods including the \textit{stochastic average gradient} (SAG) method \cite{Leroux2012sgd}, the SAGA method \cite{BachSagaMethod14} and the MISO method \cite{MairalSurrogate2013}. The applications include but are not limited to logistic regression and binary regression with $\ell_2$ regularization and more recently training conditional random fields \cite{Schmidt:15:CRF}. These methods process a single component function at a time as in incremental methods, but keeps a memory of the most recent gradients of all component functions so that an approximate gradient descent direction of $f$ 
is taken at each iteration. They might require an excessive amount of memory when $m$ is large, however they have fast convergence properties on strongly convex functions with constant stepsize without requiring the restrictive gradient growth condition. Furthermore, IAG forms approximations to the gradient of the objective function $\nabla f(x)$ at each step and this provides an accurate and efficient stopping criterion (stop if the norm of the approximate gradient is below a certain threshold) whereas it is often not clear ``when to stop" with IG. 


The IAG method was first proposed in a pioneer work by \cite{Blatt2007incremental} where its global convergence under some assumptions is shown. It is also shown that in the special case when each $f_i$ is a quadratic, IAG exhibits global linear convergence if the stepsize is small enough; however, neither an explicit convergence rate nor an explicit upper bound on the stepsize that can lead to linear convergence was given. This result is based on a perturbation analysis of the eigenvalues of a periodic dynamic linear system which is of independent interest in terms of the techniques used but is also highly technical and computationally demanding as it requires estimating the derivatives of the eigenvalues of a one-parameter matrix family. Furthermore, it only applies to quadratic functions. More recently, Tseng and Yun \cite{TsengYun2014incremental} proved global 
convergence under less restrictive conditions and local linear convergence in a more general setting when each component function satisfies a local Lipschitzian error condition, a condition  satisfied by locally strongly convex functions (around an optimal solution). Although the results are more general than those of \cite{Blatt2007incremental} as they apply beyond quadratics, the proofs are still involved and do not contain any explicit rate estimates because $(i)$ the constants involved in the analysis are implicit and hard to compute/approximate, $(ii)$ the results are asymptotic (they hold when the stepsize is small enough but bounds on the stepsize are not available). See Remark \ref{rema-compare-rate} for more detail. 

In this paper, we present a novel convergence analysis for the IAG method with several advantages and implications. First, our analysis is based on a careful choice of a Lyapunov function which leads to simple global and linear convergence proofs. Furthermore, our proofs give more insight into the behavior of IAG compared to previous approaches, showing that IAG can be treated as a perturbed gradient descent method where gradient errors can be interpreted as shocks with a finite duration that are fading away as one gets closer to an optimal solution (in a way we will make precise). Second, to our knowledge, our analysis is the first to provide explicit rate estimates for (deterministic) IAG methods. Third, we discuss an ``IAG method with momentum" and show its global and linear convergence. To our knowledge, this is the first global convergence and linear convergence result for an aggregated gradient method with memory. 

In many applications, there is a favorable deterministic order to process the functions. For instance, in source localization or parameter estimation problems over sensor networks, sensors are a part of a big network structure and are only able to communicate with their neighbors subject to certain constraints in terms of geography and distance, and this may enforce to follow a particular deterministic order (see e.g. \cite{Blatt2007incremental}). There exists other similar scenarios where the data is distributed over different units that are connected to each other in a particular fashion (e.g. connected through a ring network) where a local optimization algorithm that accesses each unit in a particular order is favorable \cite{rabbat2004distributed}. These applications motivate the study of deterministic incremental methods such as IAG which performs well in practice \cite{Blatt2007incremental}. 

There has been some recent work on the SAG algorithm, the stochastic version of the IAG method where the order is stochastic. For SAG, Le Roux \emph{et al.} \cite{Leroux2012sgd} and later Defazio \emph{et al.} \cite{BachSagaMethod14} established a global linear convergence rate in (expected cost) expectation which is a weaker averaged sense of convergence compared to the deterministic convergence we will consider in this work. In particular, our results applies to any order (as long as each function is visited at least once in a finite number of steps $K$) and hold deterministically (not in a probabilistic sense). As the performance of the determistic incremental methods are sensitive to the specific order chosen (see e.g. \cite[Example 1.5.6]{Bertsekas99nonlinear}, \cite[Figure 2.1.9]{Bertsekas15Book}), IAG can be slower than SAG if an unfavorable order is chosen and the analysis of IAG has to account for this worst-case scenarios. This is also reflected in Theorem \ref{theorem-IAG-global-lin-conv} where the rate of IAG has a worse (quadratic) dependance in the condition number $Q$ (see \eqref{def-cond-number} for a definition) whereas SAG has linear dependance \cite[Prop 1]{Leroux2012sgd}. We also note that most of the proofs and proof techniques used in the stochastic setting such as the fact that the expected gradient error is zero do not apply to the deterministic setting and this requires a new approach for analyzing IAG.

In the next section, Section 2, we describe the IAG method. Section \ref{sec-convergence} introduces the main assumptions and estimates for our convergence analysis and the linear rate result. In Section \ref{sec-iag-momentum}, we develop a new IAG method with momentum and provide a linear rate result for it. In Section \ref{sec-summary}, we conclude by discussing summary and future work.  
\paragraph{Notation} We use $\| \cdot \|$ to denote the standard Euclidean norm on $\R^n$. For a real scalar $x$, we define $(x)_+ = \max(x,0)$. The gradient and the Hessian matrix of $f$ at a point $x \in \R^n$ are denoted by $\nabla f(x)$ and $\nabla^2 f(x)$ respectively. The Euclidean (dot) inner product of two vectors $v_1,v_2\in\R^n$ is denoted by $\langle v_1, v_2\rangle$.

\section{IAG method}
For a constant stepsize $\gamma>0$, an integer $K\geq 0$ and arbitrary initial points $x^0, x^{-1}, \dots, x^{-K} \in \R^n$, the IAG method consists of the following iterations, 
	\beqa \quad g^k &=& \sum_{i=1}^m \nabla f_i(x^{\tik}),\label{iter-grad-update}\\   
           x^{k+1} &=& x^k - \gamma g^k, \quad k=0,1,2,\dots,\label{iter-x-update}
	\eeqa
where the gradient sampling times $\{\tik\}_{i=1}^m$ can be arbitrary as long as they are sampled at least once in the last $K$ iterations, i.e.
    \beq\label{ineq-bdd-delay} k \geq \tik \geq k - K, \quad i=1,2,\dots,m. 
    \eeq    
In other words, $K$ is an upper bound on the delay encountered by the gradients of the component functions. The update \eqref{iter-grad-update} determines the direction of motion $-g^k$ by approximating the steepest descent direction $-\nabla f(x^k)$ at (iterate) time $k$ from the recently computed gradients of the component functions (at times $\{\tik\}_{i=1}^m$). For example, if the component functions are processed one by one using a deterministic cyclic order on the index set $\{1,2,\dots,m\}$ with initialization $\tau_i^0=0$ for all $i$, then $\tik$ admits the recursion \beqas
        \tik = \begin{cases}
                        k, \quad \text{if $i = (k-1 \mod m) + 1$} \\
                        \tau^{k-1}_i, \quad \text{else},
                    \end{cases}, \quad 1\leq i \leq m, \quad k=1,2,\dots,
\eeqas
which satisfies $\tik \geq k-(m-1)$ for all $i,k$ where $K=m-1$. This is the original IAG method introduced by Blatt \textit{et al.}\cite{Blatt2007incremental}. Later on, Tseng and Yun \cite{TsengYun2014incremental} generalized this method by allowing more general gradient sampling times $\{\tik\}$ with bounded delays, i.e. satisfying \eqref{ineq-bdd-delay}.  

As IAG takes an approximate steepest gradient descent direction, it is natural to analyze it as a perturbed steepest gradient descent method. In fact, in the special case, when $K=0$, all the gradients are up-to-date and IAG reduces to the classical (non-incremental) gradient descent (GD) method which is well known. Therefore, the more interesting case that we will analyze is when $K$ is strictly positive. For simplicity of notation in our analysis, we will also take 
    \beq\label{eq-iag-init} x^0 = x^{-1} = \dots = x^{-K}
    \eeq
This results in the initialization $\tau_i^0=0$ for all $i$. However, it will be clear that our analysis can be extended to other (arbitrary) choices of initial points $\{x^j\}_{j=-K}^0$ in a straightforward manner. 
\section{Convergence analysis}
\label{sec-convergence}
\subsection{Preliminaries}
We will make the following assumptions that have appeared in a number of papers analyzing incremental methods including \cite{Solodov98IncrGrad}, \cite{TsengYun2014incremental}, \cite{AlgEkfs2003}, \cite{BachSagaMethod14} and \cite{Leroux2012sgd}.

\begin{assumption}\label{assum-strong-cvx-lip-grad}\textbf{(Strong convexity and Lipschitz gradients)} 
\begin{itemize}
    \item [(i)] Each $f_i$ is $L_i$-smooth meaning it has Lipschitz continuous gradients on $\R^n$ 
satisfying
   $$ \| \nabla f_i (y) - \nabla f_i(z) \| \leq L_i \|y-z\|,  \quad \forall y,z \in \R^n,$$
where $L_i \geq 0$ is the Lipschitz constant, $i=1,2,\dots,m$. Let $\hessub= \sum_{i=1}^m \hessubi$.  
   \item [(ii)] The function $f$ is strongly convex on $\R^n$ with parameter $\hesslb>0$ meaning that the function $x \mapsto f(x) - \frac{\hesslb}{2}\|x\|^2$ is convex.
\end{itemize} 
\end{assumption}
\noindent It follows by the triangle inequality that $f$ has Lipschitz continous gradients with a Lipschitz constant $\hessub$.\footnote{The Lipschitz constant $L=\sum_i L_i $ of $\nabla f$ may not be the best Lipschitz constant as Lipschitz constants $L_i$ of $\nabla f_i$ are subadditive.} We define the condition number of $f$ as
   \beq\label{def-cond-number} Q = \frac{L}{\hesslb}\geq 1
   \eeq
(see e.g.  \cite{Nesterov04CvxBook}).  As an example, in the special case when each $f_i$ is a quadratic function, $\nabla^2 f_i(x)$ is a constant (matrix) for each $i$ and we can take $L_i$ to be its largest eigenvalue whereas the strong convexity constant $c$ can be taken as the  smallest eigenvalue of the Hessian of $f$.

A consequence of Assumption \ref{assum-strong-cvx-lip-grad} on the strong convexity of $f$ is that there exists a unique optimal solution of the problem \eqref{pbm-multi-agent} which we denote by $x^*$. In addition to the strong convexity, by the gradient Lipschitzness assumption, we have
    \beqa\label{grad-lip-bd}\| \nabla f(x) \| &\leq& L \| x - x^* \|, \quad ~~\forall x \in \R^n, \\
     f(x) - f(x^*) &\leq& \frac{L}{2}\| x- x^*\|^2, \quad \forall x \in \R^n, \label{grad-lip-in-cost}
     \eeqa
(see \cite[Theorem 2.1.5]{Nesterov04CvxBook}). In addition, it follows from \cite[Theorem 2.1.12]{Nesterov04CvxBook} that    
\beqas \langle \nabla f(x), x - x^* \rangle \geq \frac{\mu L}{\mu + L} \| x- x^*\|^2 + \frac{1}{\mu + L} \|\nabla f(x)\|^2,  \quad \forall x \in \R^n. \label{str-cvx-ineq}
\eeqas      
We finally introduce the following lemma for proving the linear rate for IAG. We omit the proof due to space considerations, for a proof see Feyzmadhavian \textit{et al.} \cite{feyzmahdavian2014delayed} where this lemma is used to analyze the effect of delays in a first-order method. Coarsely speaking, the intuition behind this lemma is the following: If a non-negative sequence $\{V_k\}$ that decays to zero linearly obeying $V_{k+1} \leq p V_k$ for some $p<1$ is perturbed with an additive (noise) shock term that depends on the recent history, i.e. the shock at step $k$ is on the order of $V_\ell$ where $\ell \in [k-d(k), k]$ and $d(k)$ is the time interval (duration) of the shock, the linear convergence property can be preserved if the shocks are small enough but this comes at the expense of a degraded rate $r > p$ which is determined by the amplitude of the shocks (controlled by the parameter $q$) and the duration of the shocks. 
   
\begin{lemm}\label{lem-pert-lin-conv} Let $\{V_k\}$ be a sequence of non-negative real numbers satisfying
		$$ V_{k+1} \leq p V_k + q \max_{\left(k-d(k)\right)_+ \leq \ell \leq k} V_\ell, \quad k \geq 0,$$
for some non-negative constants $p$ and $q$. If $p+q<1$ and $0 \leq d(k) \leq  d_{\max}$
for some positive constant $d_{\max}$, then 
         $$V_k \leq r^k V_0, \quad k\geq 1,$$ 
where $r=(p+q)^{\frac{1}{1+d_{\max}}}$.            
\end{lemm}
          
\subsection{Bounding gradient error}
\label{subsec-grad-error}
We denote the distance to the optimal solution at iterate $k$ by
   \beq \dist_k = \|x^k - x^*\|
   \eeq
and the gradient error by
  \beq \label{def-iag-grad-error}
     e^k = g^k - \nabla f(x^k).
  \eeq
We will show that the gradient error can be bounded in terms of a finite sum involving distances of iterates to the optimal solution. Using the triangle inequality and the Lipschitzness of the gradients, for any $k\geq 0$, 
  \beqa \|e^k\| \leq \sum_{i=1}^m \| \nabla f_i(x^{\tik}) - \nabla f_i(x^k) \| 
   \leq   \sum_{i=1}^m \hessubi \| x^{\tik} - x^k \| \label{ineq-grad-error-dist-based}.
\eeqa
As the gradient delays are bounded by $K$ (see \eqref{ineq-bdd-delay}), by a repetitive application of the triangle inequality, we obtain for any $k\geq 0$,
\beqa   
   \|e^k\| &\leq &  \sum_{i=1}^m \hessubi \sum_{j=\tik}^{k-1}  \| x^{j+1} - x^j \| \label{ineq-grad-abs-error}
   \leq   \hessub \sum_{j=(k-K)_+}^{k-1}  \| x^{j+1} - x^j \|  \label{ineq-grad-error-dist} \\
   &=&  \gamma \hessub \sum_{j=(k-K)_+}^{k-1}  \| g^j \| 
   \leq  \gamma \hessub  \sum_{j=(k-K)_+}^{k-1} \bigg( \| \nabla f(x^j) \| + \|e^j\| \bigg) \label{grad-err-bd-of-order-step}
  \eeqa
(with the convention that $\|e^0\| = 0$ which is implied by \eqref{eq-iag-init}).
The inequality \eqref{grad-err-bd-of-order-step} provides a recursive upper bound for the gradient error that relates the gradient error $\|e^k\|$ to the previous gradient errors $\|e^j\|$ with $j<k$. Bounding the $\|e^j\|$ term in \eqref{grad-err-bd-of-order-step} with the previously obtained bound \eqref{ineq-grad-error-dist-based} on the gradient error and by the triangle inequality,
   \beqa
   	   \|e^k\| &\leq & \gamma\hessub  \sum_{j=(k-K)_+}^{k-1} \bigg( \| \nabla f(x^j) \| +   \sum_{i=1}^m \hessubi \| x^{\tau_i^j} - x^j \| \bigg)\label{ineq-iag-main-grad-error-bound-begin} \\
       &\leq& \gamma\hessub  \sum_{j=(k-K)_+}^{k-1} \bigg( \| \nabla f(x^j) \| +   \sum_{i=1}^m \hessubi \big(\| x^{\tau_i^j} - x^* \| + \| x^* - x^j \|\big) \bigg) \nonumber\\   	   
       &\leq& \gamma\hessub \sum_{j=(k-K)_+}^{k-1} \bigg( \| \nabla f(x^j)\| + \sum_{i=1}^m \hessubi \big(\max_{\ell \in \{ \tau_i^j\}_{i=1}^m } \dist_\ell + \dist_j \big) \bigg). \nonumber 
   \eeqa
Invoking \eqref{ineq-bdd-delay} on the boundedness of the gradient delays by $K$ once more and using \eqref{grad-lip-bd} on gradient Lipschitzness to bound the norm of the gradient, we finally get    
   \beqa       
   	   \|e^k\|  &\leq& \gamma\hessub \sum_{j=(k-K)_+}^{k-1} \bigg( \| \nabla f(x^j)\| + 2\hessub \max_{(j-K)_+ \leq\ell\leq j} \dist_\ell \bigg)\nonumber \\
	&\leq& \gamma\hessub  \sum_{j=(k-K)_+}^{k-1} \bigg( \hessub~\dist_j + 2\hessub \max_{(j-K)_+ \leq\ell\leq j} \dist_\ell \bigg)\nonumber \\
	&\leq& \gamma\hessub  \sum_{j=(k-K)_+}^{k-1} \bigg( 3\hessub \max_{(j-K)_+ \leq\ell\leq j} \dist_\ell \bigg) 
	\leq 3\gamma\hessub^2 K    \max_{(k-2K)_+ \leq\ell\leq k-1} \dist_\ell.  \label{ineq-iag-main-grad-error-bound}
   \eeqa
  
\subsection{Linear convergence analysis}
\label{subseq-lin-conv}
From the IAG update formul\ae ~\eqref{iter-grad-update}--\eqref{iter-x-update} and the definition \eqref{def-iag-grad-error} of the gradient error, it follows directly by taking norm squares that 
   \beqa \dist_{k+1}^2 &=& \dist_k^2 -2\gamma \langle \nabla f(x^k), x^k - x^* \rangle + \gamma^2 \|\nabla f(x^k)\|^2  + E_k  \label{error-term-dist-start}
\eeqa
where the gradient errors are incapsulated by the last term 
     \beq E_k = \gamma^2 \|e^k\|^2 - 2\gamma \langle x^k - x^* - \gamma \nabla f(x^k), e^k \rangle. \label{error-term-dist-iteration}
     \eeq 
Using the inequality \eqref{str-cvx-ineq} for strongly convex functions,
   \beqa \dist_{k+1}^2 &\leq& \bigg (1- 2\gamma \frac{\mu L}{\mu + L} \bigg) \dist_k^2 + \gamma \bigg(\gamma - \frac{2}{\mu + L}\bigg)\|\nabla f(x^k)\|^2 + E_k  \\
      &\leq& \bigg (1- 2\gamma \frac{\mu L}{\mu + L} \bigg) \dist_k^2 + E_k, \quad \mbox{if} \quad \gamma \leq  \frac{2}{\mu + L }. \label{ineq-dist-bound-with-grad-error}
\eeqa
Note that when $K=0$, IAG reduces to the GD method where $e^k = 0$ and the $E_k$ term vanishes. In this special case, the last  inequality simplifies to
    \beqa \|x^k - x^* - \gamma \nabla f(x^k)\|^2 \leq \bigg (1- 2\gamma \frac{\mu L}{\mu + L} \bigg) \dist_k^2 \leq \dist_k^2 \label{classic-gradient-method-rate} 
\eeqa
and a choice of $\gamma = \frac{2}{\mu + L }$ leads to 
the global linear convergence satisfying
   $$ \dist_{k+1} = \|x^k - x^* - \frac{2}{\mu + L } \nabla f(x^k)\| \leq \bigg(\frac{Q-1}{Q+1}\bigg) \dist_k \leq \bigg(\frac{Q-1}{Q+1}\bigg)^{k} \dist_0. 
   $$
This is the standard analysis of the GD method (see \cite[Theorem 2.1.15]{Nesterov04CvxBook}).  
%
The next theorem shows that in the more interesting general case when there are gradient errors, i.e. when $K>0$ and $e^k \neq 0$, a similar linear convergence argument can be done although one has to use a smaller stepsize to compensate for the gradient errors. The main idea is to eliminate the gradient error $\|e^k\|$ terms in \eqref{ineq-dist-bound-with-grad-error} by replacing them with terms involving only distances. This can be done by invoking \eqref{ineq-iag-main-grad-error-bound} which essentially provides an upper bound for the gradient errors in terms of distances. Then, Lemma \ref{lem-pert-lin-conv} with $V_k = \dist_k^2$ applies and provides the convergence rate. 

\begin{theo}\label{theorem-IAG-global-lin-conv} Suppose that Assumption \ref{assum-strong-cvx-lip-grad} holds. Consider the IAG iterations \eqref{iter-grad-update}--\eqref{iter-x-update} with an integer gradient delay parameter $K > 0$ and a constant stepsize $ 0 < \gamma < \bar{\gamma}$ where
   $$\bar{\gamma} = \bigg(\frac{a\mu}{KL}\bigg)\frac{1}{\mu + \hessub}$$
with $a = 8/25$.  Then, IAG iterates $\{x^k\}$ are globally linearly convergent. Furthermore, when $\gamma=\gamma_*:=\bar{\gamma}/2$ we have for $k=1,2,\dots$, 
\beqa
  \|x^k - x^*\| &\leq&  \bigg(1 -\frac{c_K }{(Q+1)^2}  \bigg)^k \|x^0 - x^*\|,  \label{conv-rate-in-dist}\\
  f(x^k) - f(x^*) &\leq& \frac{L}{2} \bigg(1 -\frac{c_K }{(Q+1)^2}  \bigg)^{2k}\|x^0 - x^*\|^2,\label{conv-rate-in-cost}
\eeqa
where  $c_K = \frac{2}{25} \bigl[K(2K+1)\bigr]^{-1}$. 
\end{theo}
\begin{proof}  As $K>0$, there are gradient delays and the error term $E_k$ defined by \eqref{error-term-dist-iteration} that appears in the evolution \eqref{ineq-dist-bound-with-grad-error} of iterates is non-zero in general. Therefore, the convergence rate is limited by how fast this error term decays. Assume $\gamma \leq \frac{2}{\mu + L}$ so that \eqref{ineq-dist-bound-with-grad-error} is applicable. Using the triangle inequality on $E_k$,  we see that
   \beqas  
        | E_k |  &\leq& \gamma^2 \|e^k\|^2 + 2\gamma \|e^k\| \| x^k - x^* - \gamma \nabla f(x^k) \| 
                  \leq \gamma^2 \|e^k\|^2 + 2\gamma \|e^k\| \dist_k  
     \eeqas
where we used \eqref{classic-gradient-method-rate} in the last step. Using \eqref{ineq-iag-main-grad-error-bound} on the gradient error, we obtain
     \beqas              
          |E_k|  &\leq& 9 \gamma^4 L^4 K^2 \max_{(k-2K)_+ \leq\ell\leq k-1} \dist_\ell^2 + 6\gamma^2 L^2 K  \max_{(k-2K)_+ \leq\ell\leq k-1} \dist_\ell  \dist_k \\
                  &\leq& \big(9 \gamma^4 L^4 K^2  + 6\gamma^2 L^2 K\big)  \max_{(k-2K)_+ \leq\ell\leq k} \dist_\ell^2. 
   \eeqas     
Plugging this bound into the recursive inequality  \eqref{ineq-dist-bound-with-grad-error} for distances leads to
   $$ \dist_{k+1}^2 \leq p(\gamma) \dist_k^2 + q(\gamma) \max_{(k-2K)_+ \leq\ell\leq k} \dist_\ell^2$$
with     
    $$  p(\gamma) =1- 2\gamma \frac{\mu L}{\mu + L} , \quad  q(\gamma) =  9 \gamma^4 L^4 K^2  + 6\gamma^2 L^2 K.$$ 
If the stepsize is small enough satisfying the condition $s(\gamma):= p(\gamma) + q(\gamma) < 1$, then by Lemma \ref{lem-pert-lin-conv} applied with $V_k = \dist_k^2$, IAG is globally linearly convergent.\footnote{Note that when $K=0$, the inequality $s(\gamma)<1$ is linear in $\gamma$, whereas when $K>0$ it is fourth order in $\gamma$ and is more restrictive.} 
Ignoring the positive $O(\gamma^4)$ term in $q(\gamma)$, this condition would require at least \beqa 1- 2\gamma \frac{\mu L}{\mu + L} + 6\gamma^2 L^2 K < 1\iff 0 < \gamma < \bigg(\frac{\mu}{3LK}\bigg)\frac{1}{\mu + L}= \frac{25}{24}\bar{\gamma}\label{thres-step-iag}
     \eeqa
which would imply 
   \beq \gamma^2 L^2 K \leq \frac{1}{9K(Q+1)^2}\leq \frac{1}{36} \label{1-over-36}
   \eeq
as both $K\geq 1$ and $Q\geq 1$. We will show that under the slightly more restrictive condition $0 < \gamma <\bar{\gamma}$, one can also handle the $O(\gamma^4)$ term as well and guarantee $s(\gamma)<1$. So assume $0 < \gamma <\bar{\gamma}$. The condition \eqref{thres-step-iag} holds and the inequality \eqref{1-over-36} is valid. This implies that 
   $$ q(\gamma) = 3\gamma^2 L^2 K (2 + 3\gamma^2 L^2 K ) \leq 3\gamma^2 L^2 K (2 +\frac{1}{12} ) \leq \frac{25}{4} \gamma^2 L^2 K $$
and after straightforward calculations that   
  \beqa s(\gamma) := p(\gamma) + q(\gamma) 
               \leq  1- 2\gamma \frac{\mu L}{\mu + L}  + \frac{25}{4} \gamma^2 L^2 K < 1. \label{ineq-quad-upper-bound-on-rate}  
  \eeqa
Then by Lemma \ref{lem-pert-lin-conv}, we have global linear convergence of the sequence $V_k = \dist_k^2$ to zero with rate $\rho(\gamma) = s(\gamma) ^{1/(2K+1)}<1$. This shows the global linear convergence of IAG. It remains to show the claimed convergence rate for $\gamma=\gamma_*$. 
Let $\gamma = \gamma_*$. Note that this minimizes the quadratic with respect to $\gamma$ in \eqref{ineq-quad-upper-bound-on-rate} leading to
   \beq\label{ineq-lin-conv-bd} s (\gamma_*) \leq 1 - \frac{4\hesslb^2}{25K (\hesslb^2 + \hessub^2)} = 1 - \frac{4}{25K (Q + 1)^2}. 
    \eeq
Then, as the linear convergence rate of the sequence $\{\dist_k^2\}$ is 
$\rho(\gamma_*)<1$, by taking square roots, the sequence $\{\dist_k\}$ is linearly convergent with rate $r_*=\rho(\gamma_*)^{1/2}$ satisfying
  $$ \dist_k \leq r_*^k \dist_0, \quad r_*= \big(s({\gamma}_*) ^{1/(2K+1)}\big)^{1/2} \leq 1 -\frac{c_K}{(Q+1)^2},  $$
where we used \eqref{ineq-lin-conv-bd} and the inequality $(1-x)^a \leq 1-ax$ for $x,a \in [0,1]$ to get an upper bound for $\bar{\rho}$. This proves the rate \eqref{conv-rate-in-dist}. Then, \eqref{conv-rate-in-cost} follows directly from  \eqref{grad-lip-in-cost} and \eqref{conv-rate-in-dist}.
\end{proof}

\begin{rema}\label{rema-compare-rate} \textbf{(Comparison with previous results)} In a related work, Tseng and Yun shows the existence of a positive constant $C$ that if the stepsize $\gamma$ is small enough then IAG has a $K$-step linear convergence rate of $\sqrt{1-C\gamma}$ under a Lipschitzian error assumption which is a strong-convexity-like condition \cite[Theorem 6.1]{TsengYun2014incremental}. However, in their analysis, there is no explicit rate estimate; as $(i)$ results are asymptotic, holding for $\gamma$ small enough without giving a precise interval for $\gamma$, $(ii)$ the constant $C$ is implicit and hard to compute/approximate as it depends on several other implicit constants and a Lipschitzian error parameter $\tau$. Our analysis is not  only simple using basic distance inequalities but also the constants are transparent and explicit.  
\end{rema}
\begin{rema} \textbf{(IAG versus IG)} Theorem \ref{theorem-IAG-global-lin-conv} shows that IAG with constant stepsize is globally linearly convergent, however the same is not true for IG. In fact, IG with constant stepsize is linearly convergent to a neighborhood of the solution but does not in general converge to the optimal solution due to the existence of gradient errors that are typically bounded away from zero \cite{Bertsekas99nonlinear}. As a consequence, achieving global convergence with IG requires to use a stepsize that goes to zero and this results in typically slow convergence \cite{bertsekas2011incremental}. In contrast, the gradient error in IAG is controlled by the distance of recent iterates to the optimal solution, therefore it is attenuated as the iterates get closer to the optimal solution and a diminishing stepsize is not needed to control the error.  

\end{rema}

\begin{rema} \textbf{(Local strong convexity implies local linear rate)} We note that when $f$ is not globally strongly convex but only locally strongly convex around a stationary point, for instance when the Hessian is not degenerate around a locally optimal solution, by a reasoning along the lines of the proof of Theorem \ref{theorem-IAG-global-lin-conv}, it is possible to show that IAG is locally linearly convergent.
\end{rema}

\section{IAG with momentum} 
\label{sec-iag-momentum}
An important variant of the GD method is the heavy-ball method \cite{Pol87} which extrapolates the direction implied by the previous two iterates by the following update rule: 
$$x^{k+1} = x^k - \gamma \nabla f(x^k) + \beta(x^k - x^{k-1})$$ 
where $\beta\geq 0$ is the \textit{momentum} parameter. It can be shown that the heavy-ball method can achieve a faster local convergence than GD when $\beta$ is in a certain range \cite[Section 3.1]{Pol87}. There has also been much interest in understanding its global convergence properties \cite{Lessard2015HeavyBall,Ghadimi2015HeavyBall}. Accelerated gradient methods introduced by Nesterov \cite{Nes1983AccGradMethod,Nesterov2007gradient} can also be thought of as \textit{momentum} methods where the momentum parameter is variable and appropriately chosen. There has been a lot of recent interest in these accelerated methods as they have optimal iteration complexity properties under some conditions \cite{Nesterov04CvxBook}.

In contrast to the recent advances in non-incremental methods with momentum, there has been less progress on incremental methods with momentum. In particular, no deterministic incremental methods with favorable convergence characteristics similar to those of accelerated gradient methods are currently known. However, there is the \textit{IG method with momentum} which consists of the \textit{inner} iterations 
   $$x_i^{k+1} = x_i^k -\gamma^k \nabla f_i(x^k) + \beta (x^k - x^{k-1}), \quad i=1,2,\dots,m \quad k\geq 1, $$
starting from $x_1^1 \in \R^n$ with the convention that $x_1^{k+1}=x_{m+1}^k$ where $\gamma^k$ is the stepsize \cite{Bertsekas1996incremental,Mangasarian1994Serial, TsengIncrGradient98}.  This method can be faster than IG on some problems especially when gradients have oscillatory behavior, however it would still require the stepsize go to zero due to gradient errors, leading to typical sublinear convergence \cite{TsengIncrGradient98}. It is natural to ask whether IAG with such an additional momentum term, which we abbreviate by \textit{IAG-M}, 
   \beq x^{k+1} = x^k - \gamma g^k + \beta(x^k-x^{k-1}), \quad k=0,1,2,\dots,\label{iter-x-update-bis}
   \eeq
would be globally convergent for $\beta$ in some range $(0,\bar{\beta})$. We expect that this algorithm can outperform IAG in problems where the individual gradients show oscillatory behavior because the momentum term provides an extra smoothing/averaging affect on the iterates. The global linear convergence of the IAG-M method for a certain range of $\beta$ values can be shown by a similar reasoning along the lines presented in Section \ref{sec-convergence}. Most of the logic in the derivation of the inequalities \eqref{ineq-grad-error-dist-based}--\eqref{ineq-iag-main-grad-error-bound} apply with the only difference that the $\|x^{j+1}-x^j\|$ terms will now contain an additional momentum term due to the modified update rule \eqref{iter-x-update-bis}. We however provide a sketch of the proof in the Appendix \ref{appendix-iag-m} for the sake of completeness.

\section{Discussion}
\label{sec-summary}
We analyzed the IAG method when component functions are strongly convex by viewing it as a gradient descent method with errors. To the best of our knowledge, our analysis provides the first explicit linear rate result. Furthermore, it is different than the existing two approaches \cite{Blatt2007incremental} and \cite{TsengYun2014incremental} in the sense that $(i)$ it is based on simple basic inequalities that makes global convergence analysis simpler, $(ii)$ gives more insight into the behavior of IAG. In particular, our analysis shows that the gradient errors can be treated as shocks with a finite duration which can be bounded in terms of distance of iterates to the optimal solution. Therefore, by choosing the stepsize small enough and using the strong convexity properties we can guarantee that the the distance to the optimal solution shrinks down at each step by a factor less than one. 

We also developed a new algorithm, IAG with momentum, and provided a linear convergence and rate analysis. It is expected that this algorithm can outperform IAG in problems where the individual gradients show oscillatory behavior, because the momentum term provides an extra (averaging) smoothing affect on the iterates. 

We note that the extension of IAG to the generalized version of \eqref{pbm-multi-agent}, $\min_{x\in\R^n} \sum_{i=1}^m f_i(x)  + h(x) $
with $h:\R^n \to \R$ convex and possibly non-smooth (such as the indicator of a function when there are constraints) is simple by an additional (proximal) step, see     \cite{TsengYun2014incremental}. Extending our linear rate results to this case may be possible and is ongoing future work.

\bibliographystyle{plain}
\bibliography{iag_refs}

\begin{thebibliography}{10}

\bibitem{Bertsekas1996incremental}
D.~Bertsekas.
\newblock Incremental least squares methods and the extended {K}alman filter.
\newblock {\em SIAM Journal on Optimization}, 6(3):807--822, 1996.

\bibitem{Bertsekas99nonlinear}
D.~Bertsekas.
\newblock {\em Nonlinear programming}.
\newblock Athena Scientific, 1999.

\bibitem{bertsekas2011incremental}
D.~Bertsekas.
\newblock Incremental gradient, subgradient, and proximal methods for convex
  optimization: a survey.
\newblock {\em Optimization for Machine Learning}, 2010:1--38, 2011.

\bibitem{Bertsekas15Book}
D.~Bertsekas.
\newblock {\em Convex Optimization Algorithms}.
\newblock Athena Scientific, 2015.

\bibitem{Blatt2007incremental}
D.~Blatt, A.~Hero, and H.~Gauchman.
\newblock A convergent incremental gradient method with a constant step size.
\newblock {\em SIAM Journal on Optimization}, 18(1):29--51, 2007.

\bibitem{BottouLecun2005}
L.~Bottou and Y.~Le~Cun.
\newblock On-line learning for very large data sets.
\newblock {\em Applied Stochastic Models in Business and Industry},
  21(2):137--151, 2005.

\bibitem{BachSagaMethod14}
Aaron Defazio, Francis Bach, and Simon Lacoste-Julien.
\newblock Saga: A fast incremental gradient method with support for
  non-strongly convex composite objectives.
\newblock In {\em Advances in Neural Information Processing Systems}, pages
  1646--1654, 2014.

\bibitem{feyzmahdavian2014delayed}
H.~R. Feyzmahdavian, A.~Aytekin, and M.~Johansson.
\newblock A delayed proximal gradient method with linear convergence rate.
\newblock In {\em Machine Learning for Signal Processing (MLSP), 2014 IEEE
  International Workshop on}, pages 1--6. IEEE, 2014.

\bibitem{Ghadimi2015HeavyBall}
E.~{Ghadimi}, H.~R. {Feyzmahdavian}, and M.~{Johansson}.
\newblock {Global convergence of the Heavy-ball method for convex
  optimization}.
\newblock {\em arXiv preprint arXiv:1412.7457}, December 2014.

\bibitem{GurOzPar15}
M.~G\"urb\"uzbalaban, A.~Ozdaglar, and P.~Parrilo.
\newblock A globally convergent incremental {N}ewton method.
\newblock {\em Mathematical Programming}, 151(1):283--313, 2015.

\bibitem{Lessard2015HeavyBall}
L.~{Lessard}, B.~{Recht}, and A.~{Packard}.
\newblock {Analysis and Design of Optimization Algorithms via Integral
  Quadratic Constraints}.
\newblock {\em arXiv preprint arXiv:1408.3595}, August 2014.

\bibitem{MairalSurrogate2013}
J.~Mairal.
\newblock {Optimization with First-Order Surrogate Functions}.
\newblock In {\em {ICML}}, volume~28 of {\em JMLR Proceedings}, pages 783--791,
  Atlanta, United States, 2013.

\bibitem{Mangasarian1994Serial}
O.~L. Mangasarian and M.~V. Solodov.
\newblock Serial and parallel backpropagation convergence via nonmonotone
  perturbed minimization.
\newblock {\em Optimization Methods and Software}, 4(2):103--116, January 1994.

\bibitem{AlgEkfs2003}
H.~Moriyama, N.~Yamashita, and M.~Fukushima.
\newblock The incremental {G}auss-{N}ewton algorithm with adaptive stepsize
  rule.
\newblock {\em Computational Optimization and Applications}, 26(2):107--141,
  2003.

\bibitem{Nes1983AccGradMethod}
Y.~Nesterov.
\newblock A method of solving a convex programming problem with convergence
  rate {$O(1/k^2)$}.
\newblock In {\em Soviet Mathematics Doklady}, volume~27, pages 372--376, 1983.

\bibitem{Nesterov04CvxBook}
Y.~Nesterov.
\newblock {\em {Introductory lectures on convex optimization: a basic course}}.
\newblock Applied Optimization. Springer, Boston, 2004.

\bibitem{Nesterov2007gradient}
Y.~Nesterov.
\newblock Gradient methods for minimizing composite objective function, 2007.
\newblock
  \hbox{http://www.ucl.be/cps/ucl/doc/core/documents/coredp2007\_76.pdf}.

\bibitem{Pol87}
B.~T. Polyak.
\newblock {\em Introduction to optimization}.
\newblock Translations series in mathematics and engineering. Optimization
  Software, Publications Division, New York, 1987.

\bibitem{rabbat2004distributed}
Michael Rabbat and Robert Nowak.
\newblock Distributed optimization in sensor networks.
\newblock In {\em Proceedings of the 3rd international symposium on Information
  processing in sensor networks}, pages 20--27. ACM, 2004.

\bibitem{Leroux2012sgd}
N.~L. Roux, M.~Schmidt, and F.R. Bach.
\newblock A stochastic gradient method with an exponential convergence rate for
  finite training sets.
\newblock In F.~Pereira, C.J.C. Burges, L.~Bottou, and K.Q. Weinberger,
  editors, {\em Advances in Neural Information Processing Systems 25}, pages
  2663--2671. Curran Associates, Inc., 2012.

\bibitem{Schmidt:15:CRF}
M.~{Schmidt}, R.~{Babanezhad}, M.~{Osama Ahmed}, A.~{Defazio}, A.~{Clifton},
  and A.~{Sarkar}.
\newblock {Non-Uniform Stochastic Average Gradient Method for Training
  Conditional Random Fields}.
\newblock {\em arXiv preprint arXiv:1504.04406}, April 2015.

\bibitem{schmidt2013fast}
M.~Schmidt and N.L. Roux.
\newblock Fast convergence of stochastic gradient descent under a strong growth
  condition.
\newblock {\em arXiv preprint arXiv:1308.6370}, 2013.

\bibitem{Solodov98IncrGrad}
M.V. Solodov.
\newblock Incremental gradient algorithms with stepsizes bounded away from
  zero.
\newblock {\em Computational Optimization and Applications}, 11(1):23--35,
  1998.

\bibitem{TsengYun2014incremental}
P.~Tseng and S.~Yun.
\newblock Incrementally updated gradient methods for constrained and
  regularized optimization.
\newblock {\em Journal of Optimization Theory and Applications},
  160(3):832--853, 2014.

\bibitem{TsengIncrGradient98}
Paul Tseng.
\newblock An incremental {Gradient(-Projection}) method with momentum term and
  adaptive stepsize rule.
\newblock {\em SIAM Journal on Optimization}, 8(2):506--531, May 1998.

\end{thebibliography}

\newpage
\begin{appendix}
\section{Proof sketch of the global linear convergence of the IAG-M method}\label{appendix-iag-m}
Using the gradient error bound \eqref{ineq-grad-error-dist} and iterate update equation \eqref{iter-x-update-bis} of IAG-M, 
	\beqas \|e^k\| = \| g^k - \nabla f(x^k) \|&\leq& \hessub \sum_{j=(k-K)_+}^{k-1}  \| x^{j+1} - x^j \|  = \hessub \sum_{j=(k-K)_+}^{k-1}  \| \gamma g^j + \beta(x^j - x^{j-1})\| \\
   &\leq& \gamma \hessub \sum_{j=(k-K)_+}^{k-1} \|  g^j \| + \beta \hessub \sum_{j=(k-K)_+}^{k-1}\|x^j - x^{j-1}\|  \\
   &\leq& \gamma \hessub \sum_{j=(k-K)_+}^{k-1} (\| \nabla f(x^j) \| + \|e^j\|) + \beta \hessub \sum_{j=(k-K)_+}^{k-1}\|x^j - x^{j-1}\|.  
   \eeqas
 Then, using \eqref{grad-lip-bd} to bound the norm of the gradient and \eqref{ineq-grad-error-dist} to bound $\|e^j\|$, this becomes 
   \beqa
  \|e^k\| &\leq& \gamma \hessub \sum_{j=(k-K)_+}^{k-1} (L \dist_j + \|e^j\|) + \beta \hessub \sum_{j=(k-K)_+}^{k-1}\|x^j - x^{j-1}\| \nonumber \\
   &\leq& \gamma \hessub \sum_{j=(k-K)_+}^{k-1} (L \dist_j +  \hessub \sum_{\ell=(j-K)_+}^{j-1}  \| x^{\ell+1} - x^\ell \| ) + \beta \hessub \sum_{j=(k-K)_+}^{k-1}\|x^j - x^{j-1}\| \nonumber \\
   &\leq& \gamma \hessub \sum_{j=(k-K)_+}^{k-1} \big(L \dist_j +  \hessub \sum_{\ell=(j-K)_+}^{j-1}  (\dist_\ell + \dist_{\ell+1})\big) + \beta \hessub \sum_{j=(k-K)_+}^{k-1}\|x^j - x^{j-1}\|  \nonumber \\
	&\leq& \big( 3\gamma \hessub^2 K  \big)  \max_{(k-2K)_+\leq\ell\leq k-1}\dist_\ell + \beta \hessub \sum_{j=(k-K)_+}^{k-1}\|x^j - x^{j-1}\|. \label{ineq-grad-error-ub-iagm}
	\eeqa
\noindent Note that when $\beta=0$, this inequality reduces to \eqref{ineq-iag-main-grad-error-bound} obtained for IAG. From the inner update equation \eqref{iter-x-update-bis}, we also have
   \beqa \|x^j - x^{j-1}\| &\leq& \gamma \|g^{j-1}\| + \beta \| x^{j-1}-x^{j-2}\| \nonumber \\
   &\leq & \gamma \|g^{j-1}\| + \beta (\| x^{j-1}-x^*\| + \| x^{j-2}-x^*\|)  \nonumber \\
   &\leq & \gamma \|\nabla f(x^{j-1})\| + \gamma \|e^{j-1}\| + \beta (\dist_{j-1} + \dist_{j-2}) \nonumber \\
    &\leq & \gamma \hessub \dist_{j-1} + \gamma \|e^{j-1}\| + 2\beta \max(\dist_{j-1},\dist_{j-2}) \label{ineq-conseq-iter-dist}    
   \eeqa
where we used \eqref{grad-lip-bd} to bound the norm of the gradient in the last inequality. We next bound the gradient error term $\|e^{j-1}\|$ on the right-hand side. A consequence of \eqref{ineq-grad-error-dist-based}, the triangle inequality and the boundedness of the gradient delays is that gradient error is bounded by
  \beqa \|e^k\|
             &\leq& \sum_{i=1}^m \hessubi (\dist_{\tik}+\dist_k) 
         \leq 2L \max_{(k-K)_+ \leq \ell\leq k} \dist_\ell, \quad k\geq 0 \label{ineq-ek-bound-simple}.
  \eeqa
Combining all the inequalities \eqref{ineq-grad-error-ub-iagm}, \eqref{ineq-conseq-iter-dist} and \eqref{ineq-ek-bound-simple} together, leads to
     \beq \|e^k\|  \leq \bigg( 3\gamma \hessub^2 K  + \beta L K (3\gamma L + 2 \beta) \bigg)  \max_{(k-2K-1)_+\leq\ell\leq k-1}\dist_\ell \label{ineq-grad-error-ub-iagm-2},
	\eeq
which is an analogue of \eqref{ineq-iag-main-grad-error-bound}	for IAG-M. Then, following the same line of argument with the proof of Theorem \ref{theorem-IAG-global-lin-conv}, it is straightforward to show a linear rate for IAG-M as long as the momentum parameter $\beta$ is not very large. More specifically, it follows that when $\beta \leq b \gamma^p$ with $p\geq 1/2$ and a positive constant $b$ and $\gamma$ is small enough, IAG-M is globally linearly convergent. The bounds on $\gamma$ and $b$ that guarantee linear convergence can also be derived in a similar way to the proof technique of Theorem \ref{theorem-IAG-global-lin-conv}. We omit the details for the sake of brevity and leave it to the reader.
\end{appendix}	 

\end{document}